\documentclass[a4paper,reqno,12pt]{amsart}
\usepackage[margin=1.2in]{geometry}
\usepackage{indentfirst}
\usepackage{amsfonts}
\usepackage{amscd}
\usepackage{amsthm}
\usepackage{amssymb}
\usepackage{amsmath}
\usepackage{thmtools}
\usepackage{epsfig}
\usepackage{graphicx}
\usepackage{slashed}
\usepackage[all,cmtip]{xy}
\usepackage{tikz}
\usepackage{tikz-cd}
\usepackage[disable]{todonotes}
\usepackage{extpfeil}
\usetikzlibrary{decorations.markings,shapes.geometric}
\usepackage{enumitem}
\usepackage{setspace}
\usepackage{hyperref}
\usetikzlibrary{matrix,arrows,positioning}

\def\itemNum$#1${\item $\displaystyle#1$
   \hfill\refstepcounter{equation}(\theequation)}

\makeatletter
\providecommand\@dotsep{5}
\renewcommand{\listoftodos}[1][\@todonotes@todolistname]{%
  \@starttoc{tdo}{#1}}
\makeatother

\newcommand{\ostar}{ 
  \mathbin{
    \mathchoice
      {\buildcircleland{\displaystyle}}
      {\buildcircleland{\textstyle}}
      {\buildcircleland{\scriptstyle}}
      {\buildcircleland{\scriptscriptstyle}}
  } 
}

\newcommand\buildcircleland[1]{%
  \begin{tikzpicture}[baseline=(X.base), inner sep=0, outer sep=0]
    \node[draw,circle] (X)  {$#1\star$};
  \end{tikzpicture}%
}

\newtheorem{Lem}{Lemma}[section]

\newtheorem*{Def}{Definition}

\theoremstyle{plain}
\newtheorem{Thm}[Lem]{Theorem}

\theoremstyle{definition}
\declaretheorem[numbered=no,name=Example,qed={\lower-0.3ex\hbox{$\triangleleft$}}]{Ex}
\newtheorem*{Rem}{Remark}
\newtheorem*{Rems}{Remarks}

\newcommand{\Hom}{\text{\textnormal{Hom}}}
\newcommand{\End}{\text{\textnormal{End}}}
\newcommand{\Ext}{\text{\textnormal{Ext}}}
\newcommand{\Aut}{\text{\textnormal{Aut}}}

\mathchardef\mhyphen="2D
\newcommand{\git}{/\!\!/}

\newcommand{\Iso}{\text{\textnormal{Iso}}}
\newcommand{\coinv}{\text{\textnormal{coinv}}}
\newcommand{\prim}{\text{\textnormal{prim}}}
\newcommand{\ev}{\text{\textnormal{ev}}}
\newcommand{\op}{\text{\textnormal{op}}}
\newcommand{\id}{\text{\textnormal{id}}}
\newcommand{\Vect}{\mathsf{Vect}}
\newcommand{\Rep}{\mathsf{Rep}}
\newcommand{\Con}{\text{\textnormal{Con}}}
\newcommand{\ind}{\text{\textnormal{ind}}}

\newcommand*{\DashedArrow}[1][]{\mathbin{\tikz [baseline=-0.25ex,-latex, dashed,#1] \draw [#1] (0pt,0.5ex) -- (1.3em,0.5ex);}}%
\newcommand*{\dashtwoheadrarrow}{\DashedArrow[->>,densely dashed]}

\makeatletter
\newbox\xrat@below
\newbox\xrat@above
\newcommand{\xrightarrowtail}[2][]{%
  \setbox\xrat@below=\hbox{\ensuremath{\scriptstyle #1}}%
  \setbox\xrat@above=\hbox{\ensuremath{\scriptstyle #2}}%
  \pgfmathsetlengthmacro{\xrat@len}{max(\wd\xrat@below,\wd\xrat@above)+.6em}%
  \mathrel{\tikz [>->,baseline=-.75ex]
                 \draw (0,0) -- node[below=-2pt] {\box\xrat@below}
                                node[above=-2pt] {\box\xrat@above}
                       (\xrat@len,0) ;}}
\makeatother

\newcommand*\circled[1]{\tikz[baseline=(char.base)]{
            \node[shape=circle,draw,inner sep=1.5pt] (char) {#1};}}

\begin{document}

\title[Green's theorem for Hall modules]{Degenerate versions of Green's theorem for \\ Hall modules}

\author[M.\,B. Young]{Matthew B. Young}
\address{Max Planck Institute for Mathematics\\
Vivatsgasse 7\\
53111 Bonn, Germany}
\email{myoung@mpim-bonn.mpg.de}

\date{\today}


\begin{abstract}
Green's theorem states that the Hall algebra of the category of representations of a quiver over a finite field is a twisted bialgebra. Considering instead categories of orthogonal or symplectic quiver representations leads to a class of modules over the Hall algebra, called Hall modules, which are also comodules. A module theoretic analogue of Green's theorem, describing the compatibility of the module and comodule structures, is not known. In this paper we prove module theoretic analogues of Green's theorem in the degenerate settings of finitary Hall modules of $\Rep_{\mathbb{F}_1}(Q)$ and constructible Hall modules of $\Rep_{\mathbb{C}}(Q)$. The result is that the module and comodule structures satisfy a compatibility condition reminiscent of that of a Yetter--Drinfeld module.
\end{abstract}

\maketitle


\setcounter{footnote}{0}

\section*{Introduction}
\addtocontents{toc}{\protect\setcounter{tocdepth}{1}}

Let $\Rep_{\mathbb{F}_q}(Q)$ be the abelian category of representations of a finite type quiver $Q$ over a finite field $\mathbb{F}_q$. Let $\mathfrak{g}_Q$ be the complex semisimple Lie algebra whose Dynkin diagram is the underlying graph of $Q$. Ringel associated to $\Rep_{\mathbb{F}_q}(Q)$ its Hall algebra $\mathcal{H}_{Q,\mathbb{F}_q}$, which he proved to be isomorphic to the positive part of the specialized quantum group $\mathcal{U}_{\sqrt{q}}^+(\mathfrak{g}_Q)$ \cite{ringel1990c}. The relationship between Hall algebras and quantum groups was later clarified by Green, who proved, firstly, that $\mathcal{H}_{Q,\mathbb{F}_q}$ is naturally a (twisted) bialgebra, a statement commonly referred to as Green's theorem, and secondly, that the isomorphism $\mathcal{H}_{Q,\mathbb{F}_q} \simeq \mathcal{U}_{\sqrt{q}}^+(\mathfrak{g}_Q)$ is one of bialgebras \cite{green1995}. More generally, a Hall algebra can be defined from any finitary abelian category and Green's theorem generalizes to assert that, if the category is hereditary, then this is a bialgebra \cite{schiffmann2012b}. Analogues of Green's theorem have been proven for other realizations of the Hall algebra. In the complex linear setting, Joyce \cite{joyce2007} associated to the (non-finitary) category $\Rep_{\mathbb{C}}(Q)$ its constructible $\mathbb{H}_Q$ and motivic $\mathbb{H}_Q^{\textnormal{mot}}$ Hall algebras. Roughly, $\mathbb{H}_Q$ can be viewed as a classical limit of $\mathbb{H}_Q^{\textnormal{mot}}$, in which the Lefschetz motive tends to one. Joyce proved that $\mathbb{H}_Q$ is a bialgebra and related $\mathbb{H}_Q$ and $\mathbb{H}_Q^{\textnormal{mot}}$ to $\mathcal{U}^+(\mathfrak{g}_Q)$ and $\mathcal{U}^+_v(\mathfrak{g}_Q)$, respectively. An analogue of Green's theorem for the cohomological Hall algebra of a symmetric quiver with potential was proven in \cite{davison2017}. In each of the above settings, the appropriate analogue of Green's theorem plays an important role in the structure theory of the Hall algebra.

Orthogonal and symplectic analogues of quiver representations have been studied by many authors \cite{sergeichuk1987}, \cite{derksen2002}, \cite{zubkov2005}. While such representations do not form an abelian category in any natural way, there is a modification of the Hall algebra construction which produces a module over the Hall algebra, called the Hall module \cite{mbyoung2018b}. In various settings, Hall modules have been shown to be related to canonical bases \cite{enomoto2009}, \cite{varagnolo2011}, representations of quantum groups \cite{mbyoung2016} and Donaldson--Thomas theory with classical structure groups \cite{mbyoung2016b}, \cite{franzenyoung2018}. While the Hall module $\mathcal{M}_{Q,\mathbb{F}_q}$ of $\Rep_{\mathbb{F}_q}(Q)$ has a natural comodule structure, the most naive analogue of Green's theorem does not hold: $\mathcal{M}_{Q,\mathbb{F}_q}$ is not a Hopf module over $\mathcal{H}_{Q,\mathbb{F}_q}$. One of the main results of \cite{mbyoung2016} (see also \cite{enomoto2009}, \cite{varagnolo2011}) is a characterization of the module-comodule compatibility when $\mathcal{M}_{Q,\mathbb{F}_q}$ is restricted to the composition subalgebra of $\mathcal{H}_{Q,\mathbb{F}_q}$. However, this characterization, being in terms of modules over Kashiwara--Enomoto's reduced $\sigma$-analogue of $\mathcal{U}_v(\mathfrak{g}_Q)$, is not easily interpreted directly in terms of the comodule structure. At present, a full description of the module-comodule compatibility is not known.

In this paper, we give a complete description of the module-comodule compatibility for two degenerate, and in some sense classical, realizations of Hall modules. The first realization is that of representations of a quiver over the so-called field with one element $\mathbb{F}_1$. The Hall algebra $\mathcal{H}_{Q, \mathbb{F}_1}$ was first studied by Szczesny \cite{szczesny2012}, who, together with collaborators, has also studied a number of other combinatorial-type Hall algebras \cite{kremnizer2009}, \cite{szczesny2012b}, \cite{szczesny2014}, \cite{eppolito2018}, \cite{szczesny2018}. See also \cite{galvez2016}. An analogue of Green's theorem holds true and plays a crucial role in these combinatorial settings. The second realization is that of constructible Hall modules of $\Rep_{\mathbb{C}}(Q)$. In both realizations, the Hall algebras and modules are equipped with degenerate coproducts and comodule structures, respectively, as opposed to their standard Hall such structures. The coproducts, for example, are cocommutative, leading to our interpretation of these objects as being classical, as opposed to quantum. Theorems \ref{thm:sdGreenClassical} and \ref{thm:sdGreenConstructible}, which are the main results of this paper, state that the Hall modules obey conditions similar to that of a Yetter--Drinfeld module, with the Hall algebra involution induced by a duality involution of the category of quiver representations playing the role of the antipode. Our results apply to more general categories than those of quiver representations. For example, we prove Theorem \ref{thm:sdGreenClassical} for finitary proto-exact categories with duality which obey a strong decomposition property. Surprisingly, this decomposition property also appears in the constructible setting. Theorem \ref{thm:sdGreenConstructible} also applies more generally, but since the background required for its proof becomes significantly longer and our primary interest is in quivers, we do not pursue this generalization.

It is a general, if at present rather vague, philosophy that $\mathcal{H}_{Q, \mathbb{F}_1}$ is a degenerate version of the $q \rightarrow 1$ limit of $\mathcal{H}_{Q,\mathbb{F}_q}$, and similarly for $\mathbb{H}_Q$ and $\mathbb{H}_Q^{\textnormal{mot}}$. It is our hope that this intuition, together with the results of this paper, can be used to further understand the module-comodule compatibility of $\mathcal{M}_{Q,\mathbb{F}_q}$ and its variants.

\subsubsection*{Acknowledgements}
During part of the preparation of this work the author was supported by a Direct Grant from the Chinese University of Hong Kong (Project No. CUHK4053289).

\section{Proto-exact categories with duality}
\addtocontents{toc}{\protect\setcounter{tocdepth}{2}}

\subsection{Proto-exact categories}

We work in the setting of proto-exact categories, a non-additive generalization of Quillen's exact categories.

\begin{Def}[{\cite[\S 2.4]{dyckerhoff2012b}}]
A proto-exact category is a pointed category $\mathcal{A}$, with zero object $0$, together with two classes of morphisms, $\mathfrak{I}$ and $\mathfrak{D}$, called inflations and deflations and denoted by $\rightarrowtail$ and $\twoheadrightarrow$, respectively, such that
\begin{enumerate}[label=(\roman*)]
\item each morphism $0 \rightarrow U$ is in $\mathfrak{I}$ and each morphism $U \rightarrow 0$ is in $\mathfrak{D}$,

\item both classes $\mathfrak{I}$ and $\mathfrak{D}$ are closed under composition and contain all isomorphisms,

\item a commutative square of the form
\begin{equation}
\label{eq:biCartDiag}
\begin{tikzpicture}[baseline= (a).base]
\node[scale=1] (a) at (0,0){
\begin{tikzcd}
U \arrow[two heads]{d} \arrow[tail]{r} & V \arrow[two heads]{d}\\
W \arrow[tail]{r} & X
\end{tikzcd}
};
\end{tikzpicture}
\end{equation}
is a pullback if and only if it is a pushout,

\item any diagram of the form $W \rightarrowtail X \twoheadleftarrow V$ can be completed to a biCartesian square of the form \eqref{eq:biCartDiag}, and

\item any diagram of the form $W \twoheadleftarrow U \rightarrowtail V$ can be completed to a biCartesian square of the form \eqref{eq:biCartDiag}.
\end{enumerate}
\end{Def}

A biCartesian square of the form \eqref{eq:biCartDiag} with $W=0$ is called a conflation.

\begin{Ex}
\begin{enumerate}[label=(\roman*)]
\item Any exact category, and hence any abelian category, has a canonical proto-exact structure. A conflation is simply a short exact sequence.

\item The category of vector spaces over $\mathbb{F}_1$ admits a (non-exact) proto-exact structure. See Section \ref{sec:f1VectSpace} for details.
\end{enumerate}
\end{Ex}

The following definitions are motivated by \cite{eppolito2018}.

\begin{Def}
Let $\mathcal{A}$ be a proto-exact category which admits finite coproducts.
\begin{enumerate}[label=(\roman*)]
\item The category $\mathcal{A}$ is called combinatorial if, for each admissible subobject $i: U \rightarrowtail X_1 \oplus X_2$, there exist admissible subobjects $i_k: U_k \rightarrowtail X_k$, $k=1,2$, and an isomorphism $f: U \rightarrow U_1 \oplus U_2$ such that $(i_1 \oplus i_2) \circ f = i$.

\item The category $\mathcal{A}$ is said to have the finite direct decomposition property if any object $W \in \mathcal{A}$ admits only finitely many decompositions $U \oplus V \simeq W$ up to isomorphism.
\end{enumerate}
\end{Def}

\addtocontents{toc}{\protect\setcounter{tocdepth}{2}}

\subsection{Categories with duality}

For background material on categories with duality, see \cite{schlichting2010}. The notation $\mathcal{A}^{\op}$ indicates the opposite of a category $\mathcal{A}$.

\begin{Def}
A category with duality is a triple $(\mathcal{A}, P, \Theta)$ consisting of a category $\mathcal{A}$, a functor $P: \mathcal{A}^{\op} \rightarrow \mathcal{A}$ and a natural isomorphism $\Theta: \id_{\mathcal{A}} \Rightarrow P \circ P^{\op}$ such that $P(\Theta_U) \circ \Theta_{P(U)} = \id_{P(U)}$ for all objects $U \in \mathcal{A}$.
\end{Def}

A (non-degenerate) symmetric form in $(\mathcal{A}, P, \Theta)$ is a pair $(N, \psi_N)$, often denoted by just $N$, consisting of an object $N \in \mathcal{A}$ and an isomorphism $\psi_N: N \rightarrow P(N)$ which satisfies $P(\psi_N) \circ \Theta_N = \psi_N$. An isometry of symmetric forms $\phi: M \rightarrow N$ is an isomorphism of the underlying objects which satisfies $\psi_M = P(\phi) \circ  \psi_N \circ \phi$. The notation $M \simeq_P N$ indicates that $M$ and $N$ are isometric. Denote by $\Aut_P(N) \leq \Aut(N) $ the group of self-isometries of $N$. The groupoid of symmetric forms in $\mathcal{A}$ and their isometries is denoted by $\mathcal{A}_{\textnormal{h}}$.

\begin{Def}
A proto-exact category with duality is a category with duality $(\mathcal{A}, P, \Theta)$ such that $\mathcal{A}$ is proto-exact and $P$ is exact, that is
\begin{enumerate}[label=(\roman*)]
\item $P(0) \simeq 0$,
\item a morphism $U \xrightarrow[]{\phi} V$ is in $\mathfrak{I}$ if and only if $P(V) \xrightarrow[]{P(\phi)} P(U)$ is in $\mathfrak{D}$, and
\item a square of the form \eqref{eq:biCartDiag} is biCartesian if and only if the square
\[
\begin{tikzpicture}[baseline= (a).base]
\node[scale=1] (a) at (0,0){
\begin{tikzcd}
P(X) \arrow[two heads]{d} \arrow[tail]{r} & P(V) \arrow[two heads]{d}\\
P(W) \arrow[tail]{r} & P(U)
\end{tikzcd}
};
\end{tikzpicture}
\]
is biCartesian.
\end{enumerate}
\end{Def}

Let $N$ be a symmetric form in $\mathcal{A}$ and let $i : U \rightarrowtail N$ be an inflation. The orthogonal $U^{\perp}$ is defined to be a kernel of the composition $P(i) \circ \psi_N$. The existence of $U^{\perp}$ is ensured by the assumption that $i$ is an inflation and $P$ is exact. The inflation $i$ is called isotropic if $P(i) \circ \psi_N \circ i =0$ and the monomorphism $U \hookrightarrow U^{\perp}$ is an inflation.

All proto-exact categories in this paper are assumed to have the following property, called the Reduction Assumption in \cite[\S 3.4]{mbyoung2018b}: If $N$ is a symmetric form and $U \rightarrowtail N$ is isotropic with orthogonal $k: U^{\perp} \rightarrowtail N$, then the quotient $U^{\perp} \slash U$, which we denote by $N \git U$, admits a symmetric form $\psi_{N \git U}$, unique up to isometry, which makes the diagram
\[
\begin{tikzpicture}[baseline= (a).base]
\node[scale=1] (a) at (0,0){
\begin{tikzcd}[column sep=3.0em, row sep=0.5em]
{} & N  \arrow{r}[above]{\psi_N} & P(N) \arrow[two heads]{dr}[above]{P(k)} & {} \\
U^{\perp} \arrow[tail]{ur}[above left]{k} \arrow[two heads]{dr}[below left]{\pi} & {} & {} & P(U^{\perp}) \\
{} & N \git U  \arrow{r}[below]{\psi_{N \git U}} & P(N \git U) \arrow[tail]{ur}[below]{P(\pi)} & {} \\
\end{tikzcd}
};
\end{tikzpicture}
\]
commute. We will use the notation
\[
U \rightarrowtail N \dashtwoheadrarrow  N \git U
\]
to indicate the $U$ is isotropic in $N$ with reduction $N \git U$, and call this a symmetric conflation. By \cite[Proposition 5.2]{quebbemann1979}, exact categories with duality satisfy the Reduction Assumption.

Suppose that $\mathcal{A}$ admits finite coproducts. In this case we will always assume that $P$ is additive. We can then make the following constructions. First, the coproduct of symmetric forms is defined by
\[
(M,\psi_M) \oplus (N, \psi_N) = (M \oplus N, \psi_M \oplus \psi_N).
\]
Secondly, given an object $U \in \mathcal{A}$, the hyperbolic symmetric form $H(U)$ on $U$ is defined to be the pair
\[
(U \oplus P(U), \left( \begin{smallmatrix} 0 & \id_{P(U)} \\ \Theta_U & 0 \end{smallmatrix} \right)).
\]
Note that there is a canonical inclusion $\Aut(U) \hookrightarrow \Aut_P(H(U))$.

We end this section with two basic results.

\begin{Lem}
\label{lem:hyperbolicIso}
Let $\mathcal{A}$ be a proto-exact category with duality which admits finite coproducts and satisfies the Reduction Assumption. Let $i: U_1 \rightarrowtail X$ be an inflation and $ \pi: X \twoheadrightarrow U_2$ a deflation. The following statements are equivalent:
\begin{enumerate}[label=(\roman*)]
\item The subobject $U_1 \oplus P(U_2) \subset H(X)$ is isotropic.

\item The inflation $U_1 \rightarrowtail X$ factors through $\ker(X \twoheadrightarrow U_2) \rightarrowtail X$.

\item The quotient $X \twoheadrightarrow U_2$ factors through $X \twoheadrightarrow X \slash U_1$.
\end{enumerate}
Moreover, if any of the above conditions hold, then there is a canonical isometry
\[
H(X) \git (U_1 \oplus P(U_2)) \simeq_P H(\ker(X \twoheadrightarrow U_2) \slash U_1).
\]
\end{Lem}

\begin{proof}
The composition
\[
U_1 \oplus P(U_2) \xrightarrowtail[]{\footnotesize \left(\begin{smallmatrix}
i & 0 \\
0 & P(\pi)
\end{smallmatrix} \right)} H(X) \xrightarrow[]{\psi_{H(X)}} P(H(X)) \xtwoheadrightarrow[]{\footnotesize \left(\begin{smallmatrix}
P(i) & 0 \\
0 & P^2(\pi)
\end{smallmatrix} \right)} P(U_1) \oplus P^2(U_2)
\]
is equal to
$\left(\begin{smallmatrix}
0 & P(\pi \circ i) \\
P^2(\pi) \circ \Theta_X \circ i & 0
\end{smallmatrix} \right)
$
and, since $P^2(\pi) \circ \Theta_X = \Theta_{U_2} \circ \pi$, is zero if and only if $\pi \circ i=0$. It follows that (i) and (ii) are equivalent. Suppose that (ii) holds, so that
\[
\begin{tikzpicture}[baseline= (a).base]
\node[scale=1] (a) at (0,0){
\begin{tikzcd}[column sep=2.0em, row sep=2.0em]
U \arrow[tail,dashed]{dr}[below left]{i^{\prime}} \arrow[tail]{r}[above]{i} & X  \arrow[two heads]{r}[above]{\pi} & U_1 \\
{} & \ker \pi . \arrow[tail]{u} & {}
\end{tikzcd}
};
\end{tikzpicture}
\]
Then the diagram
\[
\begin{tikzpicture}[baseline= (a).base]
\node[scale=1] (a) at (0,0){
\begin{tikzcd}[column sep=2.0em, row sep=2.0em]
U_1 \arrow[two heads]{d} \arrow[tail]{r}[above]{i} & X \arrow[two heads]{d} \arrow[two heads]{ddr}[above right]{\pi}& {}\\
0 \arrow[tail]{r} \arrow[tail]{drr}& X \slash U_1 & {} \\
{} & {} & U_2
\end{tikzcd}
};
\end{tikzpicture}
\]
commutes. Since the square is a pushout, the universal property yields the desired factorization of $\pi$, showing that (iii) holds. The proof that (iii) implies (ii) is similar.

Turning to the final statement, assume that $U_1 \oplus P(U_2) \subset H(X)$ is isotropic. A direct calculation shows that
\[
(U_1 \oplus P(U_2))^{\perp} \simeq \ker(X \twoheadrightarrow U_2) \oplus P(X \slash U_1).
\]
The inflation $U_1 \oplus P(U_2) \rightarrowtail (U_1 \oplus P(U_2))^{\perp}$ is determined by the maps from (ii) and (iii). The claimed isometry follows.
\end{proof}

\begin{Lem}
\label{lem:noIsoAssump}
Let $\mathcal{A}$ be a combinatorial proto-exact category with duality. Let $M_1$ and $M_2$ be symmetric forms in $\mathcal{A}$ and let $i: U \rightarrowtail M_1 \oplus M_2$ be isotropic. Then $i$ factors through isotropic inflations $i_k: U_k \rightarrowtail M_k$, $k=1,2$.
\end{Lem}

\begin{proof}
Consider the commutative diagram
\[
\begin{tikzpicture}[baseline= (a).base]
\node[scale=0.9] (a) at (0,0){
\begin{tikzcd}[column sep=3.0em, row sep=3.0em]
U \arrow{d}[left]{f} \arrow[tail]{r}[above]{i} & M_1 \oplus M_2  \arrow{r}[above]{\psi_1 \oplus \psi_2} & P(M_1) \oplus P(M_2) \arrow[two heads]{r}[above]{P(i)} \arrow[two heads]{rd}[below left]{P(i_1) \oplus P(i_2)} & P(U)\\
U_1 \oplus U_2 \arrow[tail]{ur}[below right]{i_1 \oplus i_2} & {}  & {} & P(U_1) \oplus P(U_2) \arrow{u}[right]{P(f)},
\end{tikzcd}
};
\end{tikzpicture}
\]
where the isomorphism $f$ is that guaranteed by the combinatorial assumption on $\mathcal{A}$. The composition $P(i) \circ (\psi_1 \oplus \psi_2) \circ i$ is zero by assumption. It follows that
\[
(P(i_1) \oplus P(i_2)) \circ (\psi_1 \oplus \psi_2) \circ (i_1 \oplus i_2) = (P(i_1) \circ \psi_1 \circ i_1 \oplus P(i_2) \circ \psi_2 \circ i_2),
\]
and hence $P(i_k) \circ \psi_k \circ i_k$, $k=1,2$, is also zero.
\end{proof}

\subsection{Vector spaces over \texorpdfstring{$\mathbb{F}_1$}{}}
\label{sec:f1VectSpace}

Let $\Vect_{\mathbb{F}_1}$ be the  category of finite dimensional vector spaces over $\mathbb{F}_1$. For a detailed discussion, see \cite{szczesny2012}. Objects of $\Vect_{\mathbb{F}_1}$ are finite pointed sets, with basepoint denoted by $*$. Morphisms are basepoint preserving maps $f:V \rightarrow W$ for which the restriction $f\vert_{V \backslash f^{-1}(*)}$ is injective.  The category $\Vect_{\mathbb{F}_1}$ admits the structure of a proto-exact category, inflations being the injective morphisms and deflations being the surjective morphisms for which the pre-image of any non-basepoint is a singleton. The zero object is $0=\{*\}$.

Given an inflation $U \rightarrowtail V$, its cokernel $V \slash U$ is obtained from $V$ by collapsing $U$ to $* \in V$. If $U^{\prime} \rightarrowtail V$ is a second inflation, then $U \cap U^{\prime} \rightarrowtail V$ is the inflation given by the set theoretic intersection of $U$ and $U^{\prime}$ in $V$. The coproduct $V \oplus V^{\prime}$ is defined to be the disjoint union of $V$ and $V^{\prime}$ modulo the identification of their basepoints. Contrary to the case of vector spaces over a field, note that if $U \rightarrowtail V \oplus V^{\prime}$, then
\begin{equation}
\label{eq:intersection}
U \simeq (U \cap V) \oplus (U \cap V^{\prime}).
\end{equation}
The dimension of $V \in \Vect_{\mathbb{F}_1}$ is $\dim_{\mathbb{F}_1} V = \vert V \vert -1$.

The category $\Vect_{\mathbb{F}_1}$ has an exact duality $(-)^{\vee} = \Hom_{\Vect_{\mathbb{F}_1}}(-, \{*, 1\})$. For each $V \in \Vect_{\mathbb{F}_1}$, there is a canonical isomorphism $V \xrightarrow[]{\sim} V^{\vee}$, $v \mapsto \delta_v$, where
\[
\delta_v(v^{\prime})  = 
\begin{cases}
1 & \mbox{ if } v= v^{\prime}, \\
* & \mbox{ if } v\neq v^{\prime}.
\end{cases}
\]
Under this identification, $(-)^{\vee}$ squares to the identity. The triple $(\Vect_{\mathbb{F}_1}, (-)^{\vee}, 1_{\id_{\Vect}})$ is therefore a proto-exact category with strict duality.

\subsection{Representations of quivers}
\label{sec:quiverReps}

For background on the representation theory of quivers, see \cite{schiffmann2012b} and \cite{szczesny2012} over fields and $\mathbb{F}_1$, respectively. For symmetric forms on quiver representations over fields, see \cite{derksen2002}, \cite{mbyoung2016}.

Throughout this section, we denote by $k$ either $\mathbb{F}_1$ or a field whose characteristic is different from two.

Let $\Rep_k(Q)$ be the category of finite dimensional representations of a quiver $Q=(Q_0,Q_1)$ in $\Vect_k$. Explicitly, a representation is a pair
\[
(U=\bigoplus_{i \in Q_0} U_i ,\{u_{\alpha}\}_{\alpha \in Q_1})
\]
consisting of a finite dimensional $Q_0$-graded vector space $U$ over $k$ and a $k$-linear morphism $u_{\alpha}: U_i \rightarrow U_j$ for each arrow $\alpha: i \rightarrow j$ of $Q$. The category $\Rep_k(Q)$ inherits a proto-exact structure from $\Vect_k$, which is abelian if $k$ is a field.

Let $\Lambda_Q = \mathbb{Z} Q_0$ and $\Lambda_Q^+ = \mathbb{Z}_{\geq 0} Q_0$. The dimension vector map $\mathbf{dim}_k: K_0(\Rep_k(Q)) \rightarrow \Lambda_Q$ is a surjective group homomorphism.

A contravariant involution $\sigma$ of $Q$ is the data of involutions of $Q_0$ and $Q_1$, both denoted by $\sigma$, with the property that if $\alpha: i \rightarrow j$ is an arrow in $Q$, then $\sigma(\alpha)$ is an arrow $\sigma(j) \rightarrow \sigma(i)$. When $k$ is a field, we also fix functions
\[
s: Q_0 \rightarrow \{ \pm 1 \}, \qquad
\tau: Q_1 \rightarrow \{ \pm 1 \}
\]
such that $s$ is $\sigma$-invariant and $\tau_{\alpha} \tau_{\sigma(\alpha)} = s_i s_j$ for all arrows $\alpha: i \rightarrow j$. This data induces an exact functor $P: \Rep_k(Q)^{\op} \rightarrow \Rep_k(Q)$, given on objects by
\[
P(U)_i =U_{\sigma(i)}^{\vee},
\qquad
P(u)_{\alpha}  = 
\begin{cases}
u_{\sigma(\alpha)}^{\vee} & \mbox{ if } k = \mathbb{F}_1, \\
\tau_{\alpha} u_{\sigma(\alpha)}^{\vee} & \mbox{ if } k \neq \mathbb{F}_1
\end{cases}
\]
and on morphisms by
\[
P(V \xrightarrow[]{\{\phi_i \}_{i \in Q_0}} W) = P(W) \xrightarrow[]{\{ \phi_{\sigma(i)}^{\vee}\}_{i \in Q_0}} P(V).
\]
Set also
\[
\Theta_U =
\begin{cases}
\id_U & \mbox{ if } k = \mathbb{F}_1,\\
\bigoplus_{i \in Q_0} s_i \cdot \ev_{U_i} & \mbox{ if } k \neq \mathbb{F}_1
\end{cases}
\]
with $\ev_{U_i}: U_i \rightarrow U_i^{\vee \vee}$ the evaluation isomorphism. Then $(\Rep_k(Q), P, \Theta)$ is a proto-exact category with duality. Symmetric quiver representations are by definition symmetric forms in $\Rep_k(Q)$.

When $k= \mathbb{F}_1$, coproducts and intersections of representations are formed by applying the definitions of Section \ref{sec:f1VectSpace} pointwise in $Q_0$.

\begin{Lem}
The category $\Rep_{\mathbb{F}_1}(Q)$ is combinatorial, has the finite direct decomposition property and satisfies the Reduction Assumption.
\end{Lem}

\begin{proof}
The first statement follows from equation \eqref{eq:intersection} and the fact that coproducts in $\Rep_{\mathbb{F}_1}(Q)$ are pointwise (in $Q_0$) coproducts in $\Vect_{\mathbb{F}_1}$. The second statement follows similarly from the finite direct decomposition property of $\Vect_{\mathbb{F}_1}$.

To prove that $\Rep_{\mathbb{F}_1}(Q)$ satisfies the Reduction Assumption, it suffices to show that the same is true of $\Vect_{\mathbb{F}_1}$. Observe that since objects of $\Vect_{\mathbb{F}_1}$ are canonically self-dual, the set of symmetric forms on $N \in \Vect_{\mathbb{F}_1}$ is in natural bijection with those $\pi \in \Aut_{\Vect_{\mathbb{F}_1}}(N) \simeq \mathfrak{S}_{\dim_{\mathbb{F}_1} N}$ which satisfy $\pi^2=e$. Moreover, $\pi$ and $\pi^{\prime}$ correspond to isometric symmetric forms if and only if they are conjugate. For a subobject $i : U \rightarrowtail (N, \psi_{\pi}),$ we find
\[
U^{\perp} = \{ x \in N \mid \pi(x) \not\in i(U) \}.
\]
Hence, $U$ is isotropic if and only if $U \cap \pi(U) = *$, in which case a symmetric form on $U^{\perp} \slash U$ is defined by deleting from $\pi$ all $2$-cycles in which an element of $i(U)$ appears. It is clear that this symmetric form is unique up to isometry.
\end{proof}

\section{The combinatorial case}
\label{sec:combCase}

\subsection{The Hall algebra of a proto-exact category}
\label{sec:hallAlg}

We recall some material about finitary Hall algebras. For further details, see \cite{schiffmann2012b}, \cite{dyckerhoff2012b}, \cite{szczesny2012}.

A proto-exact category $\mathcal{A}$ is called finitary if it is essentially small and the sets
\[
\Hom_{\mathcal{A}}(U,V), \qquad \Ext^1_{\mathcal{A}}(U,V)
\]
are finite for all $U,V \in \mathcal{A}$. Here $\Ext^1_{\mathcal{A}}(U,V)$ is defined via a Yoneda-type construction, as conflations up to equivalence. Let $\Iso(\mathcal{A})$ be the set of isomorphism classes of objects of $\mathcal{A}$. The Grothendieck group $K_0(\mathcal{A})$ is the quotient of the free abelian group generated by $\Iso(\mathcal{A})$ by the subgroup generated by the relations $[W] = [U] + [V]$ whenever $U \rightarrowtail W \twoheadrightarrow V$ is a conflation.

The Hall algebra of a finitary proto-exact category $\mathcal{A}$ is the rational vector space with basis $\Iso(\mathcal{A})$,
\[
\mathcal{H}_{\mathcal{A}} = \bigoplus_{U \in \Iso(\mathcal{A})} \mathbb{Q} \cdot [U].
\]
The associative product is
\[
[U] \cdot [V] = \sum_{W \in \Iso(\mathcal{A})} F^W_{U,V} [W].
\]
Here $F^W_{U,V} \in \mathbb{Z}$ is the cardinality of the set
\[
\mathcal{F}^W_{U,V} = \{ \widetilde{U} \overset{\textnormal{admiss.}}{\subset} W \mid \widetilde{U} \simeq U, \;\; W \slash \widetilde{U} \simeq V\},
\] 
which is finite since $\mathcal{A}$ is finitary. The algebra unit $1_{\mathcal{H}}$ is $[0]$. The algebra $\mathcal{H}_{\mathcal{A}}$ is canonically graded by the submonoid $K_0(\mathcal{A})^+ \subset K_0(\mathcal{A})$ of classes of objects.

Assume now that $\mathcal{A}$ admits finite coproducts. Let $\mathcal{H}_{\mathcal{A}} \widehat{\otimes}_{\mathbb{Q}} \mathcal{H}_{\mathcal{A}}$ be the vector space of formal (possibly infinite) linear combinations $\sum_{U,V} c_{U,V} [U] \otimes [V]$. Define a $\mathbb{Q}$-linear map $\Delta: \mathcal{H}_{\mathcal{A}} \rightarrow \mathcal{H}_{\mathcal{A}} \widehat{\otimes}_{\mathbb{Q}} \mathcal{H}_{\mathcal{A}}$ by the formula
\[
\Delta ([W]) = \sum_{\substack{(U,V) \in \Iso(\mathcal{A})^{\times 2} \\ U \oplus V \simeq W}} [U] \otimes [V].
\]
This makes $\mathcal{H}_{\mathcal{A}}$ into a cocommutative topological coalgebra. If $\mathcal{A}$ satisfies the finite direct decomposition property, then the image of $\Delta$ lies in $\mathcal{H}_{\mathcal{A}} \otimes_{\mathbb{Q}} \mathcal{H}_{\mathcal{A}}$ and $\mathcal{H}_{\mathcal{A}}$ is a coalgebra. We remark that $\Delta$ is not the standard Hall coproduct; see the remarks at the end of Section \ref{sec:hallModBackground}.

\begin{Rem}
For later comparison, observe that $\mathcal{H}_{\mathcal{A}}$ can be identified with the space of finitely supported functions $\Iso(\mathcal{A}) \rightarrow \mathbb{Q}$. From this point of view, the product and coproduct are given by
\[
(f_1 \cdot f_2) (W) = \sum_{U \subset W} f_1(U) f_2(W \slash U), \qquad \Delta(f)(U,V) = f(U \oplus V),
\]
respectively.
\end{Rem}

The relevance of the combinatorial property to Hall algebras is the following.

\begin{Thm}[{\cite[Theorem 2.4]{eppolito2018}}]
\label{thm:protoGreen}
Let $\mathcal{A}$ be a finitary combinatorial proto-exact category. Then $(\mathcal{H}_{\mathcal{A}}, \cdot, \Delta)$ is a topological cocommutative bialgebra.
\end{Thm}

For simplicity, we will henceforth assume that $\mathcal{A}$ has the finite direct decomposition property, so that $\mathcal{H}_Q$ is an honest bialgebra.

An element $u \in \mathcal{H}_{\mathcal{A}}$ is called primitive if it satisfies
\[
\Delta (u) = u \otimes 1_{\mathcal{H}} + 1_{\mathcal{H}} \otimes u.
\]
It follows from the definition of $\Delta$ that the set of isomorphism classes of indecomposable objects of $\mathcal{A}$ is a basis of the vector space $V_{\mathcal{A}}^{\prim}$ of primitive elements of $\mathcal{H}_{\mathcal{A}}$. The commutator bracket gives $V_{\mathcal{A}}^{\prim}$ the structure of a $K_0(\mathcal{A})^+$-graded pro-nilpotent Lie algebra, the universal enveloping algebra of which is denoted by $\mathcal{U}(V^{\prim}_{\mathcal{A}})$.

Since $\mathcal{H}_{\mathcal{A}}$ is a connected $K_0(\mathcal{A})^+$-graded cocommutative bialgebra, the Milnor--Moore theorem implies the following result.

\begin{Thm}[{\cite[Theorem 6.1]{szczesny2012}}]
\label{thm:envelHall}
The canonical map $\mathcal{U}(V^{\prim}_{\mathcal{A}}) \rightarrow \mathcal{H}_{\mathcal{A}}$ is a bialgebra isomorphism.
\end{Thm}

\subsection{Hall algebras in the presence of a duality structure}
\label{sec:hallAlgDuality}

We continue to work in the setting of Section \ref{sec:hallAlg}. Suppose that $(P, \Theta)$ is an exact duality on $\mathcal{A}$. The Hall algebra then inherits a $\mathbb{Q}$-linear map
\[
P_{\mathcal{H}} : \mathcal{H}_{\mathcal{A}} \rightarrow \mathcal{H}_{\mathcal{A}}, \qquad 
[U] \mapsto [P(U)].
\]
The isomorphism $\id_{\mathcal{A}} \simeq P \circ P^{\op}$ implies that $P_{\mathcal{H}}$ squares to the identity. Since $P$ is contravariant and exact, we have the equalities
\[
F^W_{U,V} = F^{P(W)}_{P(V),P(U)}, \qquad U,V,W \in \mathcal{A}
\]
which imply that $P_{\mathcal{H}}$ is an algebra anti-homomorphism. As $P$ is additive, $P_{\mathcal{H}}$ is a coalgebra homomorphism.

\begin{Lem}
\label{lem:lieAlgAntiinv}
The map $P_{\mathcal{H}}: \mathcal{H}_{\mathcal{A}} \rightarrow \mathcal{H}_{\mathcal{A}}$ restricts to a Lie algebra anti-involution $P_{\mathcal{H}} : V_{\mathcal{A}}^{\prim} \rightarrow V_{\mathcal{A}}^{\prim}$.
\end{Lem}

\begin{proof}
That $P_{\mathcal{H}}$ restricts to a linear map $V_{\mathcal{A}}^{\prim} \rightarrow V_{\mathcal{A}}^{\prim}$ follows from the fact that $P_{\mathcal{H}}$ is a coalgebra homomorphism. That this map is a Lie algebra anti-homomorphism follows from the fact that $P_{\mathcal{H}}$ is an algebra anti-homomorphism:
\[
P_{\mathcal{H}}([x,y]) = P_{\mathcal{H}}(xy) - P_{\mathcal{H}}(yx) = P_{\mathcal{H}}(y)P_{\mathcal{H}}(x) - P_{\mathcal{H}}(x) P_{\mathcal{H}}(y) = -[P_{\mathcal{H}}(x),P_{\mathcal{H}}(y)].
\]
\end{proof}

Denote by $V^{\prim,\pm}_{\mathcal{A}} \subset V^{\prim}_{\mathcal{A}}$ the subspace of $P_{\mathcal{H}}$-invariants ($+$) or $P_{\mathcal{H}}$-anti-invariants ($-$). There is a vector space splitting $V^{\prim}_{\mathcal{A}} = V^{\prim,+}_{\mathcal{A}} \oplus V^{\prim,-}_{\mathcal{A}}$.

\begin{Lem}
\label{lem:lieSubalg}
\begin{enumerate}[label=(\roman*)]
\item The subspace $V^{\prim,-}_{\mathcal{A}} \subset V^{\prim}_{\mathcal{A}}$ is a Lie subalgebra.

\item The subspace $V^{\prim,+}_{\mathcal{A}} \subset V^{\prim}_{\mathcal{A}}$ is naturally a representation of   $V^{\prim,-}_{\mathcal{A}}$.
\end{enumerate}
\end{Lem}

\begin{proof}
The first statement follows from a calculation similar to that from Lemma \ref{lem:lieAlgAntiinv}. The second statement follows from the observation that $V_{\mathcal{A}}^{\prim,+}$ is stable under the restriction of the adjoint action of $V_{\mathcal{A}}^{\prim}$ to the Lie subalgebra $V^{\prim,-}_{\mathcal{A}}$.
\end{proof}

\subsection{The Hall module of a proto-exact category with duality}
\label{sec:hallModBackground}

Let $\mathcal{A}$ be a finitary proto-exact category with duality which satisfies the Reduction Assumption. Let $GW_0(\mathcal{A})$ be the Grothendieck--Witt group of $\mathcal{A}$, that is, the free abelian group generated by the set $\Iso_{\textnormal{h}}(\mathcal{A})$ of isometry classes of symmetric forms, modulo the subgroup generated by relation $[N] = [H(L)]$ whenever $L \rightarrowtail N$ is a Lagrangian.

Following \cite[\S 2.2]{mbyoung2016}, the Hall module of $(\mathcal{A}, P, \Theta)$ is the rational vector space with basis $\Iso_{\textnormal{h}}(\mathcal{A})$,
\[
\mathcal{M}_{\mathcal{A}}= \bigoplus_{M \in \Iso_{\textnormal{h}}(\mathcal{A})} \mathbb{Q} \cdot [M].
\]
Define a left $\mathcal{H}_{\mathcal{A}}$-action on $\mathcal{M}_{\mathcal{A}}$ by the formula
\[
[U] \star [M] = \sum_{N \in \Iso_{\textnormal{h}}(\mathcal{A})} G^N_{U,M} [N],
\]
where $G^N_{U,M} \in \mathbb{Z}$ is the cardinality of the set
\begin{equation}
\label{eq:sdHallSet}
\mathcal{G}^N_{U,M} = \{ \widetilde{U} \overset{\textnormal{admiss.}}{\subset} N \mid \widetilde{U} \simeq U,\, \widetilde{U} \subset N \mbox{ is isotropic},\;\; N \git \widetilde{U} \simeq_P M\}.
\end{equation}
The assumption that $\mathcal{A}$ is finitary ensures that $\mathcal{G}^N_{U,M}$ is finite; see \cite[\S 2.2]{mbyoung2016}. The above formula makes $\mathcal{M}_{\mathcal{A}}$ into a left $\mathcal{H}_{\mathcal{A}}$-module. This can be verified directly, as in \cite[Theorem 2.4]{mbyoung2016}, or as a consequence of the relative $2$-Segal perspective of \cite[\S 3.4]{mbyoung2018b}. The vector space $\mathcal{M}_{\mathcal{A}}$ is graded by the submonoid $GW_0(\mathcal{A})^+ \subset GW_0(\mathcal{A})$ of classes of symmetric forms. The module structure $\star$ is compatible with the $K_0(\mathcal{A})^+$ and $GW_0(\mathcal{A})^+$ gradings via the monoid homomorphism
\[
K_0(\mathcal{A})^+ \rightarrow GW_0(\mathcal{A})^+, \qquad [U] \mapsto [H(U)].
\]

Suppose now that $\mathcal{A}$ admits finite copoducts and that $P$ is additive. Define a $\mathbb{Q}$-linear map
\[
\rho : \mathcal{M}_{\mathcal{A}} \rightarrow \mathcal{H}_{\mathcal{A}} \widehat{\otimes}_{\mathbb{Q}} \mathcal{M}_{\mathcal{A}},
\qquad
[N] \mapsto \sum_{\substack{(U,M) \in \Iso(\mathcal{A}) \times \Iso_{\textnormal{h}}(\mathcal{A}) \\ H(U) \oplus M \simeq_P N}}  [U] \otimes [M].
\]
This makes $\mathcal{M}_{\mathcal{A}}$ into a topological left $\mathcal{H}_{\mathcal{A}}$-comodule. Direct inspection shows that the equality
\begin{equation}
\label{eq:hallModCocomm}
(P_{\mathcal{H}} \otimes \id_{\mathcal{M}}) \circ \rho = \rho
\end{equation}
holds. Equivalently, for any $\xi \in \mathcal{M}_{\mathcal{A}}$, the terms $[U] \otimes [M]$ and $[P(U)] \otimes [M]$ appear in $\rho(\xi)$ with equal coefficients. The equality \eqref{eq:hallModCocomm} can be viewed as a module-theoretic analogue of the cocommutativity of $\mathcal{H}_{\mathcal{A}}$.

\begin{Rem}
There is also an interpretation of $\mathcal{M}_{\mathcal{A}}$ as the space of finitely supported functions $\Iso_{\textnormal{h}}(\mathcal{A}) \rightarrow \mathbb{Q}$ together with a convolution action of $\mathcal{H}_{\mathcal{A}}$.
\end{Rem}

By the following lemma, if $\mathcal{A}$ has the finite direct decomposition property, then $\mathcal{M}_{\mathcal{A}}$ is an honest comodule.

\begin{Lem}
If $\mathcal{A}$ has the finite direct decomposition property, then any symmetric form $N$ in $\mathcal{A}$ admits at most finitely many orthogonal decompositions $H(U) \oplus M \simeq_P N$, up to isomorphism in $U$ and isometry in $M$.
\end{Lem}

\begin{proof}
This follows from the observation that, because of the finitary assumption, any object of $\mathcal{A}$ admits at most finitely many symmetric forms.
\end{proof}

It is easy to see that $\mathcal{M}_{\mathcal{A}}$ is not a Hopf module over $\mathcal{H}_{\mathcal{A}}$, that is, $\rho$ is not a module homomorphism over $\Delta$. For example, this is already the case for $\mathcal{A}=\mathsf{Vect}_{\mathbb{F}_1}$.

\begin{Rems}
\begin{enumerate}[label=(\roman*)]
\item Heuristically, the coproduct $\Delta$ from Section \ref{sec:hallAlg} is a degenerate version of the Hall coproduct, as defined in \cite[\S 1]{green1995}, \cite[\S 1.4]{schiffmann2012b}, in that the latter is a sum over all conflations with fixed middle term, while the former involves only trivial (coproduct) conflations. In the same way, $\rho$ is a degenerate version of the Hall comodule structure, as defined in \cite[\S 2.1]{mbyoung2016}, which sums over all symmetric conflations with fixed middle term. That $\Delta$ is commutative and equality \eqref{eq:hallModCocomm} holds depends strongly on the use of degenerate coalgebra/comodule structures.

\item Recall that Green's theorem states that the Hall algebra of a finitary hereditary abelian category, with its Hall coproduct, is a twisted topological bialgebra \cite[Theorem 1]{green1995}; see also \cite[Theorem 1.9]{schiffmann2012}. In view of the previous remark, we will regard Theorem \ref{thm:protoGreen} as a degenerate version of Green's theorem.
\end{enumerate}
\end{Rems}

\subsection{A degenerate Green's theorem in the combinatorial case}

In this section we prove a module-comodule compatibility result for $\mathcal{M}_{\mathcal{A}}$.

To begin, consider $\mathcal{H}_{\mathcal{A}}$ as an $\mathcal{H}_{\mathcal{A}} \mhyphen \mathcal{H}_{\mathcal{A}}$-bimodule via left and right multiplication. By using the algebra anti-involution $P_{\mathcal{H}}$, we can view $\mathcal{H}_{\mathcal{A}}$ as a left $\mathcal{H}_{\mathcal{A}}^{\otimes 2}$-module. In this way, $\mathcal{H}_{\mathcal{A}}\otimes_{\mathbb{Q}} \mathcal{M}_{\mathcal{A}}$ becomes a left $\mathcal{H}_{\mathcal{A}}^{\otimes 3}$-module. Denote this module structure by $\star_P$. Explicitly, we have
\begin{equation}
\label{eq:PAction}
([U_1] \otimes [U_2] \otimes [U_3]) \star_P ([U_{-1}] \otimes [N]) = [U_1] \cdot [U_{-1}] \cdot [P(U_2)] \otimes [U_3] \star [N]
\end{equation}
where $U_i \in \Iso(\mathcal{A})$ and $N \in \Iso_{\textnormal{h}}(\mathcal{A})$.

We now come to the main result of this section.

\begin{Thm}
\label{thm:sdGreenClassical}
Let $\mathcal{A}$ be a finitary combinatorial proto-exact category with duality which has the finite direct decomposition property. Then the equality
\begin{equation}
\label{eq:classicalModCompat}
\rho(u \star \xi) = \Delta^2 (u) \star_P \rho(\xi)
\end{equation}
holds for all $u \in \mathcal{H}_{\mathcal{A}}$ and $\xi \in \mathcal{M}_{\mathcal{A}}$.
\end{Thm}

\begin{proof}
By linearity, it suffices to prove that equation \eqref{eq:classicalModCompat} holds when $u=[U]$ and $\xi =[M]$. A direct calculation gives
\[
\rho([U] \star [M]) = \sum_{\substack{V \in \Iso(\mathcal{A}) \\ R \in \Iso_{\textnormal{h}}(\mathcal{A})}} G^{H(V) \oplus R}_{U,M} [V] \otimes [R].
\]
Similarly, if we denote by $\Upsilon_{U,M}$ the set of tuples
\[
\{ X,Y,Z, W \in \Iso(\mathcal{A}), \, T \in \Iso_{\textnormal{h}}(\mathcal{A}) \mid X \oplus Y \oplus Z \simeq U, \; H(W) \oplus T \simeq_P M \},
\]
then we can write
\begin{eqnarray*}
\Delta^2 ([U]) \star_P \rho([M]) & = & 
\sum_{(W, X,Y,Z; T) \in \Upsilon_{U,M}} ([X] \otimes [Y] \otimes [Z]) \star_P ([W] \otimes [T]) \\
& = & \sum_{\substack{(A,V) \in \Iso(\mathcal{A})^{\times 2} \\ R \in \Iso_{\textnormal{h}}(\mathcal{A})}} \sum_{(W, X,Y,Z; T) \in \Upsilon_{U,M} } F^V_{X,A} F^A_{W,P(Y)}  G^R_{Z,T} [V] \otimes [R].
\end{eqnarray*}
To prove the theorem it therefore suffices to define, for each $U,V \in \Iso(\mathcal{A})$ and $M,R \in \Iso_{\textnormal{h}}(\mathcal{A})$, a bijection
\[
\Phi: \mathcal{G}^{H(V) \oplus R}_{U,M} \longrightarrow \bigsqcup_{\substack{(W, X,Y,Z; T) \in \Upsilon_{U,M} \\ A \in \Iso(\mathcal{A})}} \mathcal{F}^V_{X,A} \times \mathcal{F}^A_{W,P(Y)} \times \mathcal{G}^R_{Z,T}.
\]

Suppose then that $\widetilde{U} \in \mathcal{G}_{U,M}^{H(V) \oplus R}$. The combinatorial assumption on $\mathcal{A}$ implies that there exist inflations $X \rightarrowtail V$, $Y \rightarrowtail P(V)$ and $Z \rightarrowtail R$ and an isomorphism
\[
\widetilde{U} \simeq X \oplus Y \oplus Z
\]
under which the map $\widetilde{U} \rightarrowtail H(V) \oplus R$ is determined by the above inflations. Lemma \ref{lem:noIsoAssump} implies that the maps $X \oplus Y \rightarrowtail H(V)$ and $Z \rightarrowtail R$ are isotropic. Let $A$ be a cokernel of $X \rightarrowtail V$. Then, by the first part of Lemma \ref{lem:hyperbolicIso}, the inflation $Y \rightarrowtail P(V)$ factors as
\[
\begin{tikzpicture}[baseline= (a).base]
\node[scale=1] (a) at (0,0){
\begin{tikzcd}[column sep=2.0em, row sep=2.0em]
Y \arrow[tail]{dr} \arrow[tail]{r} & P(A) \arrow[tail]{d} \\
{} & P(V).
\end{tikzcd}
};
\end{tikzpicture}
\]
Let $W \rightarrowtail A$ be a kernel of $A \twoheadrightarrow P(Y)$. Applying the final statement of Lemma \ref{lem:hyperbolicIso}, we find an isometry
\[
H(V) \git (X \oplus Y) \simeq_P H(W).
\]
Using this, we compute
\[
H(W) \oplus T \simeq_P H(V) \git (X \oplus Y) \oplus R \git Z \simeq_P (H(V) \oplus R) \git U \simeq_P M.
\]
It follows that the assignment $U \mapsto (X, W, Z)$ defines a map $\Phi$ as above.

We now construct an inverse $\Psi$ of $\Phi$. Let
\[
(\widetilde{X}, \widetilde{W}, \widetilde{Z}) \in \mathcal{F}^V_{X,A} \times \mathcal{F}^A_{W,P(Y)} \times \mathcal{G}^R_{Z,T}
\]
be an element of the codomain of $\Phi$ and define $\widetilde{Y} = P(A \slash \widetilde{W})$. Put
\begin{equation}
\label{eq:invBij}
\widetilde{U} = \widetilde{X} \oplus \widetilde{Y} \oplus \widetilde{Z} \rightarrowtail V \oplus P(A) \oplus R \rightarrowtail H(V) \oplus R,
\end{equation}
where, to define the second map, we have fixed an identification $A \simeq V \slash X$ and hence an inflation $P(A) \rightarrowtail P(V)$. The composition $\widetilde{X} \oplus \widetilde{Y} \oplus \widetilde{Z} \rightarrowtail H(V) \oplus R$  is isotropic since $\widetilde{Z} \rightarrowtail R$ is so and $\widetilde{X} \oplus \widetilde{Y} \rightarrowtail H(V)$ satisfies the conditions of Lemma \ref{lem:hyperbolicIso}. Using the final part of Lemma \ref{lem:hyperbolicIso}, we compute
\begin{eqnarray*}
(H(V) \oplus R) \git (\widetilde{X} \oplus \widetilde{Y} \oplus \widetilde{Z}) & \simeq_P & H(V) \git (\widetilde{X} \oplus \widetilde{Y}) \oplus R \git \widetilde{Z} \\
& \simeq_P &  H(W) \oplus T \\
& \simeq_P &  M.
\end{eqnarray*}
We can therefore define $\Psi(\widetilde{X}, \widetilde{Y}, \widetilde{Z})$ to be the composed inflation \eqref{eq:invBij}.

It is straightforward to verify that $\Psi$ is the inverse of $\Phi$.
\end{proof}

\begin{Rem}
Theorem \ref{thm:sdGreenClassical} continues to hold without the finite direct decomposition assumption. The only difference is that a short argument using the finitary assumption is required to verify that the right hand side of equation \eqref{eq:classicalModCompat} is well-defined.
\end{Rem}

Let $(H,\cdot, \Delta)$ be a Hopf algebra with antipode $S$ and let $(M,\star, \rho)$ be a left $H$-module and a left $H$-comodule. If the module and comodule structures satisfy the compatibility condition
\[
\rho(h \star m) = \Delta^2 (h) \star_{YD} \rho(m), \qquad h \in H, \, m \in M
\]
where, in Sweedler notation, the right hand side is defined to be
\[
(h_{(1)} \otimes h_{(2)} \otimes h_{(3)}) \star_{YD} (m_{(-1)} \otimes m_{(0)}) = h_{(1)} \cdot m_{(-1)} \cdot S(h_{(3)}) \otimes h_{(2)} \star m_{(0)},
\]
then $M$ is called a Yetter--Drinfeld module \cite{radford1993}. If $H$ is cocommutative, then $\Delta^2 (h) \star_{YD} \rho(m)$ can be rewritten as
\[
h_{(1)} \cdot m_{(-1)} \cdot S(h_{(2)}) \otimes h_{(3)} \star m_{(0)}.
\]
In this way, Theorem \ref{thm:sdGreenClassical} shows that $\mathcal{M}_{\mathcal{A}}$ has a structure reminiscent of a Yetter--Drinfeld module over $\mathcal{H}_{\mathcal{A}}$, with $P_{\mathcal{H}}$ playing the role of the antipode. Note that while $P_{\mathcal{H}}$ is an anti-homomorphism of both algebras and coalgebras, it is not an antipode.

\subsection{Representations of the primitive Lie algebra}
\label{sec:lieAlgReps}

As an application of Theorem \ref{thm:sdGreenClassical}, we construct additional algebraic structures on the space of coinvariants.

We begin with a standard definition from the theory of comodules.

\begin{Def}
An element of the left $\mathcal{H}_{\mathcal{A}}$-comodule $\xi  \in \mathcal{M}_{\mathcal{A}}$ is called coinvariant if it satisfies $\rho(\xi) = 1_{\mathcal{H}} \otimes \xi$.
\end{Def}

Write $W^{\coinv}_{\mathcal{A}}$ for the $GW_0(\mathcal{A})^+$-graded vector space of coinvariants of $\mathcal{M}_{\mathcal{A}}$.

\begin{Lem}
\label{lem:cuspSubspace}
The set of isometry classes of symmetric forms which do not admit a non-trivial hyperbolic summand is a homogeneous basis of $W^{\coinv}_{\mathcal{A}}$.
\end{Lem}

\begin{proof}
From the definition of $\rho$, it is immediate that if $M$ is a symmetric form which does not admit a non-trivial hyperbolic summand, then $[M] \in \mathcal{M}_{\mathcal{A}}$ is coinvariant. Suppose then that
\[
\xi = \sum_{M \in \Iso_{\textnormal{h}}(\mathcal{A})} a_M [M] \in \mathcal{M}_{\mathcal{A}}
\]
is coinvariant. Assume that $a_{M_0}$ is non-zero. If $M_0$ admits a non-trivial hyperbolic summand, say $H(U) \oplus N \simeq_P M_0$ for some non-zero $U \in \mathcal{A}$ and (possibly trivial) $N \in \mathcal{A}_{\textnormal{h}}$, then $[U] \otimes [N]$ appears with coefficient $a_{M_0}$ in $\rho(\xi)$. This contradicts the assumption that $\xi$ is coinvariant.
\end{proof}

The following result is an analogue of the fact that the vector space of primitive elements of a bialgebra is a Lie algebra under the commutator bracket.

\begin{Thm}
\phantomsection
\label{thm:primLieAction}
\hspace{2em}
\begin{enumerate}[label=(\roman*)]
\item Let $u \in V^{\prim}_{\mathcal{A}}$ and $\xi \in W^{\coinv}_{\mathcal{A}}$. Then $u \star \xi \in W^{\coinv}_{\mathcal{A}}$ if and only if $u \in V^{\prim,-}_{\mathcal{A}}$.

\item The left action map
\[
L: V^{\prim,-}_{\mathcal{A}} \rightarrow \End_{\mathbb{Q}} (W^{\coinv}_{\mathcal{A}}), \qquad u \mapsto (\xi \mapsto u \star \xi)
\]
gives $W^{\coinv}_{\mathcal{A}}$ the structure of a representation of the Lie algebra $V^{\prim,-}_{\mathcal{A}}$.
\end{enumerate}
\end{Thm}

\begin{proof}
Let $u \in V^{\prim}_{\mathcal{A}}$ and $\xi \in W^{\coinv}_{\mathcal{A}}$. Then we have
\[
\Delta^2(u) = u \otimes 1_{\mathcal{H}} \otimes 1_{\mathcal{H}} + 1_{\mathcal{H}} \otimes u \otimes 1_{\mathcal{H}} + 1_{\mathcal{H}} \otimes 1_{\mathcal{H}} \otimes u,
\qquad
\rho(\xi) = 1_{\mathcal{H}} \otimes \xi.
\]
Using Theorem \ref{thm:sdGreenClassical}, we compute
\begin{eqnarray*}
\rho(u \star \xi) & = & \Delta^2(u) \star_P \rho(\xi) \\
& = & (u \otimes 1_{\mathcal{H}} \otimes 1_{\mathcal{H}} + 1_{\mathcal{H}} \otimes u \otimes 1_{\mathcal{H}} + 1_{\mathcal{H}} \otimes 1_{\mathcal{H}} \otimes u) \star_P (1 \otimes \xi) \\
& = & (u + P_{\mathcal{H}}(u)) \otimes \xi +  1_{\mathcal{H}} \otimes u \star \xi.
\end{eqnarray*}
This proves the first statement.

By the first statement, the left action of $V_{\mathcal{A}}^{\prim} \subset \mathcal{H}_{\mathcal{A}}$ on $\mathcal{M}_{\mathcal{A}}$ restricts to an action of $V^{\prim,-}_{\mathcal{A}}$ on $W^{\coinv}_{\mathcal{A}}$. A direct calculation then shows that this is a Lie algebra representation.
\end{proof}

\begin{Rem}
\begin{enumerate}[label=(\roman*)]
\item The second part of Theorem \ref{thm:primLieAction} should be contrasted with the case of a Hopf module $M$ over a Hopf algebra $H$, where there appears to be no natural non-trivial action of $V_H^{\prim}$ on $W^{\coinv}_M$.

\item On the other hand, one can show, using the above arguments, that if $M$ is a Yetter--Drinfeld module over a Hopf algebra $H$ with antipode $S$, then $W^{\coinv}_M$ is a representation of the Lie subalgebra
\[
V^{\prim,-}_H = \{v \in V^{\prim}_H \mid S(v)=-v\} \subset V^{\prim}_H.
\]
\end{enumerate}
\end{Rem}

\section{The constructible case}
\label{sec:constructCase}

While all results of this section remain valid over an algebraically closed field of characteristic zero, we will for simplicity take the ground field to be $\mathbb{C}$. Since some of the arguments in this section are similar to those of Section \ref{sec:combCase}, we will at points be brief.

\subsection{Constructible functions}
\label{sec:constructFun}

We refer the reader to \cite[\S\S 3,4]{joyce2006b} for a comprehensive discussion of constructible functions on schemes and stacks.

Let $Y$ be a complex algebraic variety. Given a subvariety $Z \subset Y$, denote by $\mathbf{1}_Z: Y \rightarrow \mathbb{C}$ the characteristic function of $Z$. Recall that a function $f: Y \rightarrow \mathbb{C}$ is said to be constructible if it is of the form
\begin{equation}
\label{eq:constrFun}
f = \sum_{i=1}^n a_i 1_{Y_i}
\end{equation}
for some constants $a_1, \dots, a_n \in \mathbb{C}$ and locally closed subvarieties $Y_1, \dots, Y_n \subset Y$. Constructible functions on $Y$ form a complex vector space $\Con(Y)$. The integral of a constructible function written in the form \eqref{eq:constrFun} is defined to be
\begin{equation}
\label{eq:varPushfwd}
\int_Y f d \chi = \sum_{i=1}^n a_i \chi(Y_i).
\end{equation}
Here $\chi(Y_i)$ denotes Euler characteristic of $Y_i$ with respect to the classical topology.

More generally, for an Artin stack $\mathfrak{Y}$, which we always assume to be locally of finite type and with affine stabilizers, Joyce \cite[\S 4]{joyce2006b} defined a complex vector space $\Con(\mathfrak{Y})$ of constructible functions on $\mathfrak{Y}$. For example, if a linear algebraic group $\mathsf{G}$ acts on an algebraic variety $Y$ with associated quotient stack $[Y \slash \mathsf{G}]$, then $\Con(\left[ Y \slash \mathsf{G} \right])$ is naturally identified with the subspace of $\mathsf{G}$-invariant constructible functions on $Y$.

Constructible functions on Artin stacks enjoy a number of functorial properties. First, if $\phi: \mathfrak{Y} \rightarrow \mathfrak{Z}$ is a finite type morphism of Artin stacks, then there is a pullback map
\[
\phi^* : \Con(\mathfrak{Z}) \rightarrow \Con(\mathfrak{Y}).
\]
See \cite[Definition 5.5]{joyce2006b}. The composition of two finite type morphisms satisfies
\begin{equation}
\label{eq:pullFunct}
(\phi_2 \circ \phi_1)^* = \phi_1^* \circ \phi_2^*.
\end{equation}
Second, if $\rho: \mathfrak{Y} \rightarrow \mathfrak{Z}$ is a representable morphism of Artin stacks, then a weighted fibrewise application of equation \eqref{eq:varPushfwd} gives rise to a pushforward map
\[
\rho_! : \Con(\mathfrak{Y}) \rightarrow \Con(\mathfrak{Z}).
\]
We refer to \cite[Definition 5.1]{joyce2006b}, where $\rho_!$ is denoted by $\textnormal{CF}^{\textnormal{stk}}(\rho)$, for a precise definition of the weights. The composition of two representable morphisms satisfies
\begin{equation}
\label{eq:pushFunct}
(\rho_2 \circ \rho_1)_! = \rho_{2!} \circ \rho_{1!}.
\end{equation}
See \cite[Theorem 5.4]{joyce2006b}. Moreover, given a homotopy pullback square
\[
\begin{tikzpicture}[baseline= (a).base]
\node[scale=1] (a) at (0,0){
\begin{tikzcd}[column sep=2.5em, row sep=2.5em]
\mathfrak{X} \arrow{d}[left]{\phi_1} \arrow{r}[above]{\rho_1} & \mathfrak{Y} \arrow{d}[right]{\phi_2} \\
\mathfrak{Z} \arrow{r}[below]{\rho_2}& \mathfrak{W}
\end{tikzcd}
};
\end{tikzpicture}
\]
in which $\phi_1$ and $\phi_2$ are finite type and $\rho_1$ and $\rho_2$ are representable, the equality
\begin{equation}
\label{eq:beckCond}
\phi_2^* \circ \rho_{2!} = \rho_{1!} \circ \phi_1^*
\end{equation}
of linear maps $\Con(\mathfrak{Z}) \rightarrow \Con(\mathfrak{Y})$ holds. See \cite[Theorem 5.6]{joyce2006b}.

\subsection{Constructible Hall algebras}
\label{sec:constructHallAlg}

Constructible versions of the Hall algebra of $\Rep_{\mathbb{C}}(Q)$ have been studied by many authors \cite{riedtmann1994}, \cite{ringel1992b}, \cite{kapranov2000b}. A detailed study of constructible and motivic Hall algebras of abelian categories can be found in \cite{joyce2007}.

Given $d \in \Lambda_Q^+$, let $R_d$ be the affine space of representations of $Q$ of dimension vector $d$. Explicitly, we have
\[
R_d \simeq \bigoplus_{\left( i \xrightarrow[]{\alpha} j \right) \in Q_1 } \Hom_{\mathbb{C}}(\mathbb{C}^{d_i}, \mathbb{C}^{d_j}).
\]
The linear algebraic group
\[
\mathsf{GL}_d = \prod_{i \in Q_0} \mathsf{GL}_{d_i}(\mathbb{C})
\]
acts linearly on $R_d$ by change of basis. The moduli stack of representations of dimension vector $d$ is the quotient stack $\mathfrak{M}_d = \left[ R_d \slash \mathsf{GL}_d \right]$. The disjoint union
\[
\mathfrak{M} = \bigsqcup_{d \in \Lambda_Q^+} \mathfrak{M}_d
\]
is an Artin stack, locally of finite type with affine stabilizers.

As a complex vector space, the constructible Hall algebra of $Q$ is
\[
\mathbb{H}_Q = \Con(\mathfrak{M}) =\bigoplus_{d \in \Lambda_Q^+} \Con(\mathfrak{M}_d).
\]
An associative algebra structure is defined on $\mathbb{H}_Q$ using the usual push-pull construction of Hall algebras; see \cite[Theorem 4.3]{joyce2007}. The product can be written in the following explicit form (\textit{cf.} Section \ref{sec:hallAlg}):
\[
(f_1 \cdot f_2)(W) = \int_{U \in \textnormal{Gr}(W)} f_1(U) f_2(W \slash U) d \chi, \qquad f_1, f_2 \in \mathbb{H}_Q.
\]
Here $\textnormal{Gr}(W)$ is the variety parameterizing subrepresentations of the representation $W$. Turning to coproducts, define a $\mathbb{C}$-linear map
\[
\Delta: \mathbb{H}_Q \rightarrow \Con(\mathfrak{M} \times \mathfrak{M})
\]
by the formula
\[
\Delta(f)(U , V) = f(U \oplus V), \qquad f \in \mathbb{H}_Q.
\]
The map $\Delta$ is coassociative and cocommutative.

The proof of \cite[Theorem 4.17]{joyce2007} shows that $\Delta$ is multiplicative, in the sense that the equality
\[
\Delta (f_1 \cdot f_2) = \Delta(f_1) \odot \Delta(f_2), \qquad f_1, f_2 \in \mathbb{H}_Q
\]
holds, where $\odot$ is the natural Hall-type product on $\Con(\mathfrak{M} \times \mathfrak{M})$. Interpreting $\Con(\mathfrak{M} \times \mathfrak{M})$ as a topological completion of $\mathbb{H}_Q \otimes_{\mathbb{C}} \mathbb{H}_Q$, multiplicativity of $\Delta$ becomes the statement that $\mathbb{H}_Q$ is a topological bialgebra. Alternatively, $\Delta$ restricts to an honest coproduct on $\mathbb{H}_Q^{\ind}$, the subalgebra of $\mathbb{H}_Q$ generated by the subspace of constructible functions on $\mathfrak{M}$ which are supported on indecomposable representations. This makes $\mathbb{H}_Q^{\ind}$ into a cocommutative bialgebra \cite[Theorem 4.20]{joyce2007}. The algebra $\mathbb{H}_Q^{\ind}$ is connected (see \cite[\S 4.5]{bridgeland2012b}) with primitive Lie algebra $\mathbb{V}^{\prim}_Q$ spanned by constructible functions supported on indecomposables. An application of the Milnor--Moore theorem then shows that there is a bialgebra isomorphism $\mathcal{U}(\mathbb{V}^{\textnormal{prim}}_Q) \xrightarrow[]{\sim} \mathbb{H}^{\ind}_Q$.

\subsection{Constructible Hall modules}
\label{sec:constructHallMod}

Suppose now that $Q$ has a contravariant involution $\sigma$ and duality data $(s, \tau)$. Let $(P,\Theta)$ be the associated duality structure on $\Rep_{\mathbb{C}}(Q)$, as in Section \ref{sec:quiverReps}.

Arguing as in Section \ref{sec:hallAlgDuality}, there is a $\mathbb{C}$-linear involution $P_{\mathbb{H}}: \mathbb{H}_Q \rightarrow \mathbb{H}_Q$ which is an algebra and coalgebra anti-involution. Since $U \in \Rep_{\mathbb{C}}(Q)$ is indecomposable if and only if $P(U)$ is indecomposable, $P_{\mathbb{H}}$ preserves $\mathbb{H}^{\ind}_Q \subset \mathbb{H}_Q$.

For each $e \in \Lambda_Q^{\sigma, +}$, there is an affine subvariety $R_e^{\sigma} \subset R_e$ of symmetric representations of dimension vector $e$ and an algebraic subgroup $\mathsf{G}_e^{\sigma} \leq \mathsf{GL}_e$ which acts linearly on $R_e^{\sigma}$ in such a way that the moduli stack of symmetric representations of dimension vector $e$ is $\mathfrak{N}_e = \left[ R_e^{\sigma} \slash \mathsf{G}_e^{\sigma} \right]$. Explicit descriptions of $R^{\sigma}_e$ and $\mathsf{G}_e^{\sigma}$ can be found in \cite[\S 1.2]{franzenyoung2018}. Set
\[
\mathfrak{N} = \bigsqcup_{e \in \Lambda_Q^{\sigma,+}} \mathfrak{N}_e.
\]
For each $(d , e) \in \Lambda_Q^+ \times \Lambda_Q^{\sigma,+}$, there is a subvariety $R_{d,e}^{\sigma} \subset R^{\sigma}_{d + H(e)}$ of symmetric representations for which the standard $Q_0$-graded isotropic subspace of dimension vector $d$ forms a subrepresentation. The subgroup $\mathsf{G}_{d,e}^{\sigma} \leq \mathsf{G}_{H(d) + e}^{\sigma}$ which stabilizes this isotropic subspace acts on $R_{d,e}^{\sigma}$. The quotient stack $\mathfrak{N}_{d,e} = \left[ R_{d,e}^{\sigma} \slash \mathsf{G}_{d,e}^{\sigma} \right]$ parametrizes symmetric short exact sequences in $\Rep_{\mathbb{C}}(Q)$ of the appropriate dimension vector. Setting
\[
\mathfrak{N}_{\bullet,\bullet} = \bigsqcup_{(d,e) \in \Lambda_Q^+ \times \Lambda_Q^{\sigma,+}} \mathfrak{N}_{d,e},
\]
we obtain a correspondence of stacks
\begin{equation}
\label{eq:hallModCorr}
\begin{tikzpicture}[baseline= (a).base]
\node[scale=1] (a) at (0,0){
\begin{tikzcd}[column sep=4.0em, row sep=0.5em]
{} & \mathfrak{N}_{\bullet,\bullet} \arrow{dl}[above left]{\pi_1 \times \pi_3^{\sigma}} \arrow{dr}[above right]{\pi_2^{\sigma}} & {} \\
\mathfrak{M} \times \mathfrak{N} & {} & \mathfrak{N},
\end{tikzcd}
};
\end{tikzpicture}
\end{equation}
with $\pi_1 \times \pi_3^{\sigma}$ and $\pi_2^{\sigma}$ sending a symmetric short exact sequence to its first and last terms and middle term, respectively.

\begin{Lem}
\phantomsection
\label{lem:finRepMaps}
\begin{enumerate}[label=(\roman*)]
\item The maps $\pi_1 \times \pi_3^{\sigma}$ and $\pi_2^{\sigma}$ are finite type.
\item The map $\pi_2^{\sigma}$ is representable.
\end{enumerate}
\end{Lem}

\begin{proof}
The first statement follows from the fact that $\mathfrak{N}_{\bullet,\bullet}$ is a disjoint union of algebraic stacks of finite type. The second statement follows from the observation that $\pi_2^{\sigma}$ is induced by a map $R_{d,e}^{\sigma} \hookrightarrow R_{H(d) + e}^{\sigma}$ which is equivariant with respect to the inclusion $\mathsf{G}_{d,e}^{\sigma} \hookrightarrow \mathsf{G}_{H(d)+e}^{\sigma}$.
\end{proof}

As a complex vector space, the constructible Hall module of $Q$ is
\[
\mathbb{M}_Q = \Con(\mathfrak{N}) = \bigoplus_{e \in \Lambda_Q^{\sigma,+}} \textnormal{Con}(\mathfrak{N}_e).
\]
By Lemma \ref{lem:finRepMaps}, push-pull along the correspondence \eqref{eq:hallModCorr} can be used to define a linear map $\mathbb{H}_Q \otimes_{\mathbb{C}} \mathbb{M}_Q \rightarrow \mathbb{M}_Q$. This makes $\mathbb{M}_Q$ into a left $\mathbb{H}_Q$-module. The module structure can be written explicitly as
\[
(f \star \xi) (N) = \int_{U \in \textnormal{IGr}(N)} f(U) \xi( N \git U) d \chi, \qquad f \in \mathbb{H}_Q, \, \xi \in \mathbb{M}_Q
\]
where $\textnormal{IGr}(N)$ is the variety parametrizing isotropic subrepresentations of the symmetric representation $N$.

\begin{Ex}
Given $U \in R_d$, denote by $\mathbf{1}_U \in \mathbb{H}_Q$ the characteristic function of the orbit $\mathsf{GL}_d \cdot U$. Similarly, associated to each $M \in R_e^{\sigma}$ is a characteristic function $\mathbf{1}^{\sigma}_M \in \mathbb{M}_Q$. In this setting, the action $\star$ specializes to
\[
\mathbf{1}_U \star \mathbf{1}^{\sigma}_M = \sum_{N \in \textnormal{Iso}_{\textnormal{h}}(\Rep_{\mathbb{C}}(Q))} \chi(\mathcal{G}^N_{U,M}) \mathbf{1}^{\sigma}_N,
\]
where $\mathcal{G}^N_{U,M}$ denotes the set \eqref{eq:sdHallSet} with its natural variety structure.
\end{Ex}

Continuing, define a $\mathbb{C}$-linear map
\[
\rho: \mathbb{M}_Q \rightarrow \Con(\mathfrak{M} \times \mathfrak{N})
\]
by the formula
\[
\rho(\xi)(U, M) = \xi (H(U) \oplus M), \qquad \xi \in \mathbb{M}_Q.
\]
This makes $\mathbb{M}_Q$ into a topological left $\mathbb{H}_Q$-comodule. It is immediate that the analogue of equation \eqref{eq:hallModCocomm} holds in the present setting.

\subsection{A degenerate Green's theorem in the constructible case}

We now come to the constructible analogue of Theorem \ref{thm:sdGreenClassical}.

\begin{Thm}
\label{thm:sdGreenConstructible}
The equality
\begin{equation}
\label{eq:constructModCompat}
\rho(u \star \xi) = \Delta^2 (u) \ostar_P \rho(\xi)
\end{equation}
holds for all $u \in \mathbb{H}_Q$ and $\xi \in \mathbb{M}_Q$, where $\ostar_P$ denotes the action of $\Con(\mathfrak{M}^{\times 3})$ on $\Con(\mathfrak{M} \times \mathfrak{N})$ determined by equation \eqref{eq:PAction}.
\end{Thm}

\begin{proof}
Let $\Upsilon$ be the Artin stack\footnote{This stack should be compared with the target set of the map $\Phi$ from the proof of Theorem \ref{thm:sdGreenClassical}.} whose complex points are pairs of diagrams of (symmetric) short exact sequences of the form
\[
\spadesuit= 
\left(
\begin{tikzpicture}[baseline= (a).base]
\node[scale=1] (a) at (0,0){
\begin{tikzcd}[column sep=1.0em, row sep=1.0em]
{} & X \arrow[tail]{d} & {} \\
{} &  V \arrow[two heads]{d} & {} \\
W \arrow[tail]{r} & A \arrow[two heads]{r} & P(Y)
\end{tikzcd}
};
\end{tikzpicture}
\;\; , \;\;
\begin{tikzpicture}[baseline= (a).base]
\node[scale=1] (a) at (0,0){
\begin{tikzcd}[column sep=1.5em, row sep=1.0em]
Z \arrow[tail]{d} \\
R \arrow[two heads,dashed]{d} \\
T
\end{tikzcd}
};
\end{tikzpicture}
\right).
\]
An isomorphism $(f,g) : \spadesuit \rightarrow \spadesuit^{\prime}$ of two such pairs consists of an isomorphism $f: V \xrightarrow[]{\sim} V^{\prime}$ and an isometry $g: R \xrightarrow[]{\sim} R^{\prime}$ such that $g(Z) = Z^{\prime}$, $f(X) = X^{\prime}$ and $f(W) = W^{\prime}$. The stack $\Upsilon$ fits into the following diagram:
\begin{equation}
\label{eq:constYDdiag}
\begin{tikzpicture}[baseline= (a).base]
\node[scale=1] (a) at (0,0){
\begin{tikzcd}[column sep=8.0em, row sep=1.75em]
\mathfrak{M} \times \mathfrak{N} & \mathfrak{M}^{\times 3} \times (\mathfrak{M} \times \mathfrak{N}) \arrow{l}[above]{\oplus^2 \times \left( H(-) \oplus (-) \right)}\\
\mathfrak{N}_{\bullet,\bullet} \arrow{d}[left]{\pi^{\sigma}_2} \arrow{u}[left]{\pi_1 \times \pi_3^{\sigma}} & \Upsilon \arrow{d}[right]{f_3} \arrow{l}[above]{f_1} \arrow{u}[right]{f_2} \\
\mathfrak{N} & \mathfrak{M} \times \mathfrak{N}. \arrow{l}[below]{H(-) \oplus (-)}
\end{tikzcd}
};
\end{tikzpicture}
\end{equation}
At the level of objects, the morphisms $f_1$, $f_2$ and $f_3$ are defined by
\[
f_1(\spadesuit)= \left( X \oplus Y \oplus Z \rightarrowtail H(V) \oplus R \dashtwoheadrarrow H(W) \oplus T \right)
\]
and
\[
f_2(\spadesuit) = ((X,Y,Z),(W,T)), \qquad
f_3(\spadesuit) = (V,R).
\]
The morphisms $f_1$ and $f_2$ are of finite type while $f_3$ is representable. This can be proved in the same way as Lemma \ref{lem:finRepMaps}. Applying $\Con(-)$ to the diagram \eqref{eq:constYDdiag} gives the following diagram:
\begin{equation}
\label{eq:constYDdiagCon}
\begin{tikzpicture}[baseline= (a).base]
\node[scale=1] (a) at (0,0){
\begin{tikzcd}[column sep=8.0em, row sep=1.75em]
\mathbb{H}_Q \otimes_{\mathbb{C}} \mathbb{M}_Q \arrow{d}[left]{(\pi_1 \times \pi_3^{\sigma})^*} \arrow{r}[above]{\left(\oplus^2 \times \left( H(-) \oplus (-) \right) \right)^*} & \Con \left( \mathfrak{M}^{\times 3} \times (\mathfrak{M} \times \mathfrak{N}) \right) \arrow{d}[right]{f_2^*} \\
\Con(\mathfrak{N}_{\bullet,\bullet}) \arrow{d}[left]{\pi^{\sigma}_{2!}} \arrow{r}[above]{f^*_1} & \Con(\Upsilon) \arrow{d}[right]{f_{3!}} \\
\mathbb{M}_Q \arrow{r}[below]{\left(H(-) \oplus (-) \right)^*} & \Con(\mathfrak{M} \times \mathfrak{N}).
\end{tikzcd}
};
\end{tikzpicture}
\end{equation}
Inspection of the definitions shows that the counterclockwise and clockwise compositions along the outer square of the diagram \eqref{eq:constYDdiagCon} give the operations on the left and right hand sides of equation \eqref{eq:constructModCompat}, respectively. Our task is therefore to establish the commutativity of the diagram \eqref{eq:constYDdiagCon}.

The upper square of the diagram \eqref{eq:constYDdiag} is homotopy commutative by inspection. Equation \eqref{eq:pullFunct} then implies that the upper square of the diagram \eqref{eq:constYDdiagCon} commutes. On the other hand, while the bottom square of \eqref{eq:constYDdiag} is homotopy commutative, it is not a homotopy pullback. Hence, we cannot apply equation \eqref{eq:beckCond} to conclude the commutativity of the bottom square of \eqref{eq:constYDdiagCon}. Instead, we argue using a modification of the proof of \cite[Theorem 4.17]{joyce2007}. Let $\mathfrak{F}$ be the homotopy pullback of the bottom left hand corner of the diagram \eqref{eq:constYDdiag}. We then have a homotopy commutative diagram
\[
\begin{tikzpicture}[baseline= (a).base]
\node[scale=1] (a) at (0,0){
\begin{tikzcd}[column sep=1.5em, row sep=1.5em]
\mathfrak{N}_{\bullet,\bullet} \arrow{dd}[left]{\pi_2^{\sigma}} & {} & \Upsilon \arrow{dd}[right,name=B]{f_3} \arrow{dl}[above left]{\phi} \arrow{ll}[above,name=A]{f_1} \\
{} & \mathfrak{F} \arrow{ul}[below left]{p_1} \arrow{dr}[below left]{p_2} \arrow[dl,phantom,"\tiny \circled{1}",midway] \arrow[rightarrow,to=A,phantom,"\tiny \circled{2}",midway] \arrow[rightarrow,to=B,phantom,"\tiny \circled{3}",midway] & {} \\
\mathfrak{N} & {} & \arrow{ll}[below]{H(-) \oplus (-)}\mathfrak{M} \times \mathfrak{N}.
\end{tikzcd}
};
\end{tikzpicture}
\]
Explicitly, the stack $\mathfrak{F}$ parametrizes tuples, abusively denoted by $(U,N,V,R)$, which consist of a symmetric short exact sequence
\[
U \rightarrowtail N \dashtwoheadrarrow M,
\]
a pair $(V,R) \in \Rep_{\mathbb{C}}(Q) \times \Rep_{\mathbb{C}}(Q)_h$ and an isometry $H(V) \oplus R \simeq_P N$. Upon applying $\Con(-)$, the square $\tiny \circled{1}$ and the triangle $\tiny \circled{2}$ commute by equations \eqref{eq:beckCond} and \eqref{eq:pullFunct}, respectively. The image of the morphism $\phi$ is the closed substack $\mathfrak{G} \subset \mathfrak{F}$ which parametrizes tuples $(U,N,V,R)$ as above which satisfy the additional condition
\[
(U \cap V) \oplus (U \cap P(V) ) \oplus (U \cap R) \simeq U.
\]
As $\mathfrak{G}$ is closed, there is a decomposition
\[
\Con(\mathfrak{F}) = \Con(\mathfrak{G}) \oplus \Con(\mathfrak{F} \backslash \mathfrak{G}).
\]
Upon restriction to $\Con(\mathfrak{G})$, the triangle $\tiny \circled{3}$ commutes because of equation \eqref{eq:pushFunct}. Consider instead the restriction of the triangle $\tiny \circled{3}$ to $\Con(\mathfrak{F} \backslash \mathfrak{G})$. By construction, $\phi^*$ restricts to the zero map; we will show that the same is true of $p_{2!}$. Let $(U,N,V,R)$ be a complex point of $\mathfrak{F} \backslash \mathfrak{G}$. Then the weight function $m_{p_2} : \mathfrak{F} \rightarrow \mathbb{Q}$ used to define the pushforward $p_{2!}$ is equal to
\[
m_{p_2}(U,N,V,R) = \chi(\Aut_{\mathfrak{M} \times \mathfrak{N}}(N,V,R) \slash \Aut_{\mathfrak{F}}(U,N,V,R)).
\]
Concretely, $\Aut_{\mathfrak{M} \times \mathfrak{N}}(N,V,R) \simeq \Aut(V) \times \Aut_P(R)$, which we interpret as a subgroup of $\Aut_P(N)$ via the embedding
\[
\Aut(V) \times \Aut_P(R) \hookrightarrow  \Aut_P(H(V)) \times \Aut_P(R) \hookrightarrow \Aut_P(N).
\]
Similarly, $\Aut_{\mathfrak{F}}(U,N,V,R)$ is the subgroup of isometries of $N$ of the above form which, in addition, preserve $U \rightarrowtail N$. Define an injective group homomorphism
\[
\mathfrak{i} : \mathbb{C}^{\times} \rightarrow \Aut_{\mathfrak{M} \times \mathfrak{N}}(N, V,R),
\qquad
z \mapsto
\left( \begin{smallmatrix} z \cdot \id_{V} & 0 \\ 0 & z^{-1}  \cdot \id_{P(V)} \end{smallmatrix} \right) \oplus \id_R.
\]
Since the image of $\mathfrak{i}$ intersects $\Aut_{\mathfrak{F}}(U,N,V,R)$ trivially, $\mathfrak{i}$ defines a free $\mathbb{C}^{\times}$-action on the variety $\Aut_{\mathfrak{M} \times \mathfrak{N}}(N,V,R) \slash \Aut_{\mathfrak{F}}(U,N,V,R)$, whence $m_{p_2}(U,N,V,R) =0$. This implies that $p_{2!}$ restricts to the zero map on $\Con(\mathfrak{F} \backslash \mathfrak{G})$.
\end{proof}

The above construction can be modified to obtain honest comodules which satisfy equation \eqref{eq:constructModCompat}. Let $\mathbb{M}_Q^{\ind}$ be the $\mathbb{H}_Q^{\ind}$-submodule of $\mathbb{M}_Q$ generated by the subspace $\mathbb{W}^{\coinv}_Q \subset \mathbb{M}_Q$ of $\mathbb{H}_Q$-coinvariants. As in Lemma \ref{lem:cuspSubspace}, the vector space $\mathbb{W}^{\coinv}_Q$ is spanned by constructible functions supported on symmetric representations which admit no non-trivial hyperbolic summands.

\begin{Thm}
The subspace $\mathbb{M}^{\ind}_Q$ is a left module and comodule over $\mathbb{H}^{\ind}_Q$ which satisfies equation \eqref{eq:constructModCompat}.
\end{Thm}

\begin{proof}
Suppose that $f_1, \dots, f_n \in \mathbb{V}^{\prim}_Q$ and $\xi \in \mathbb{W}^{\coinv}_Q$. By Theorem \ref{thm:sdGreenConstructible}, we have
\[
\rho((f_1 \cdots f_n) \star \xi)
=
\Delta^2(f_1 \cdots f_n) \ostar_P (1_{\mathbb{H}} \otimes \xi).
\]
The primitive and coinvariant assumptions give
\[
\Delta^2(f_i) = f_i \otimes 1_{\mathbb{H}} \otimes 1_{\mathbb{H}} + 1_{\mathbb{H}} \otimes f_i \otimes 1_{\mathbb{H}} +1_{\mathbb{H}} \otimes 1_{\mathbb{H}} \otimes f_i
\]
and $\rho(\xi)= 1_{\mathbb{H}} \otimes \xi$, respectively. Hence $\Delta^2(f_1 \cdots f_n) \in \mathbb{H}_Q^{\otimes 3} \subset \Con(\mathfrak{M}^{\times 3})$ and $\rho(\xi) \in \mathbb{H}_Q \otimes_{\mathbb{C}} \mathbb{M}_Q \subset \Con(\mathfrak{M} \times \mathfrak{N})$. We can therefore replace $\ostar_P$ with $\star_P$. As $\mathbb{H}^{\ind}_Q$ is a bialgebra and is $P_{\mathbb{H}}$-stable, we then easily see that
\[
\rho((f_1 \cdots f_n) \star \xi) \in \mathbb{H}^{\ind}_Q \otimes_{\mathbb{C}}  \mathbb{M}^{\ind}_Q \subset \Con(\mathfrak{M} \times \mathfrak{N}).
\]
Since elements of the form $(f_1 \cdots f_n) \star \xi$ span $\mathbb{M}^{\ind}_Q$, we conclude that $\mathbb{M}^{\ind}_Q$ is an $\mathbb{H}^{\ind}_Q$-comodule. The remaining statements follow from Theorem \ref{thm:sdGreenConstructible}.
\end{proof}

The results of Section \ref{sec:lieAlgReps} continue to hold in the constructible setting, with essentially the same proofs.

\bibliographystyle{plain}
\bibliography{mybib}
 

\end{document}